\documentclass[12pt]{amsart}

\usepackage[dvipsnames]{xcolor}
\usepackage[unicode,
  colorlinks=true,
  linktocpage=true,
  citecolor=Emerald,
  linkcolor=Fuchsia,
  urlcolor=Fuchsia,
  pdftitle={Components and realizability of fusion systems},
  pdfauthor={Ellen Henke and Justin Lynd},
  pdfkeywords={fusion system, component}]{hyperref}
\usepackage{microtype}
\usepackage{color}

\theoremstyle{plain}
\newtheorem{theorem}{Theorem}
\newtheorem{proposition}{Proposition}
\newtheorem{lemma}{Lemma}

\theoremstyle{definition}
\newtheorem{definition}[theorem]{Definition}

\newtheorem{example}[theorem]{Example}

\usepackage{todo}

\newcommand{\C}{\mathcal{C}}
\newcommand{\D}{\mathcal{D}}
\newcommand{\E}{\mathcal{E}}
\newcommand{\F}{\mathcal{F}}
\newcommand{\K}{\mathcal{K}}
\newcommand{\G}{\mathcal{G}}

\newcommand{\Aut}{\operatorname{Aut}}
\newcommand{\id}{\operatorname{id}}

\newcommand{\la}{\langle\!\langle}
\newcommand{\ra}{\rangle\!\rangle}
\newcommand{\norm}{\trianglelefteq}
\newcommand{\subn}{{\norm\norm}}
\newcommand{\<}{\langle}
\renewcommand{\>}{\rangle}

\title[Components and realizability of fusion systems]{Components and realizability of fusion systems}

\author[E.~Henke]{Ellen Henke}
\address{Institut f{\"u}r Algebra, Fakult{\"a}t Mathematik, Technische Universit{\"a}t Dresden, 01062 Dresden, Germany}
\email{ellen.henke@tu-dresden.de}
\author[J.~Lynd]{Justin Lynd}
\address{Department of Mathematics, University of Louisiana at Lafayette, Lafayette, LA 70504, USA}
\email{lynd@louisiana.edu}

\thanks{The authors would like to thank the Mathematisches Forschungsinstitut Oberwolfach for its hospitality during the workshop ``Finite groups, fusion systems and applications'', where the initial ideas for this project were formed. J.L. warmly thanks the Institute of Algebra at TU Dresden for its support of his visit in March 2025. 
E.H. was funded by the Deutsche Forschungsgemeinschaft (DFG, German Research
Foundation) – Projektnummer 511577973.}

\begin{document}

\begin{abstract}
For $p\in\{2,3\}$ it is known that a saturated $p$-fusion system is realizable if and only if each of its components is realizable by a finite simple group. For primes $p\geq 5$ this is false. Building on work of Broto, M{\o}ller, Oliver and Ruiz, we show however that a fusion system $\F$ is realizable if and only if for each of its components $\C$ there exists a realizable subnormal subsystem $\E$ of $\F$ with $O^{p^\prime}(\E)=\C$.  
\end{abstract}

\maketitle

\section{Introduction}

The theory of saturated fusion systems generalizes important aspects of finite group theory as each finite group $G$ leads to a fusion system $\F_S(G)$ which encodes the conjugacy relations between subgroups of a fixed Sylow $p$-subgroup $S$. It is a major objective in the field to understand under which conditions a fusion system is realizable by a finite group and under which conditions it is not (thereby called exotic). Building on the classification of finite simple groups, a recent fundamental result of Broto, M{\o}ller, Oliver and Ruiz \cite[Corollary~B]{BrotoMollerOliverRuiz2023} shows that a fusion system is realizable by a finite group if each of its components is realizable. For $p=2$, the converse was known before by a Lemma of Aschbacher (cf. \cite[Lemma~2.54]{HenkeLynd2022}), again assuming the classification of finite simple groups. As we explain below, the converse is also true at the prime $3$. So when $p \in\{2,3\}$, a fusion system is realizable if and only if each of its components is realizable. For primes $p\geq 5$ this is false, as there are examples of finite simple groups $G$ for which $O^{p^\prime}(\F_S(G))$ is an exotic quasisimple fusion system. 
A comprehensive list of such finite simple groups is given in \cite[Section~4]{OliverRuiz2021}. It was first shown by Ruiz \cite{Ruiz2007} that this phenomenon occurs for certain projective special linear groups for a careful choice of the field and the prime $p$. 

\smallskip

Thus, even if a saturated fusion system $\F=\F_S(G)$ is realizable by a finite group $G$ with Sylow $p$-subgroup $S$, it is not necessarily the case that every component is realizable. Nevertheless, if we choose $G$ to be $p'$-reduced, there is a strong connection between the components of $\F$ and the components of $G$. Namely, for each component $\C$ of $\F$ there exists a component $K$ of $G$ such that $\C=O^{p^\prime}(\F_{S\cap K}(K))$. This is shown in \cite[Proposition~4.9]{BrotoMollerOliverRuiz2023}, but we also restate this result as Theorem~\ref{T:Realizable} below and supply an alternative proof. For $p=3$, it follows from \cite[Theorem~4.5]{BrotoMollerOliverRuiz2023} that $\C=O^{p^\prime}(\F_{S\cap K}(K))$ is always realizable by a quasisimple group. Thus, as remarked above, every component of a realizable $3$-fusion system is indeed realizable. 

\smallskip

For primes $p\geq 5$, the situation is a bit more subtle. In general, if $K$ is subnormal in a finite group $G$ with Sylow $p$-subgroup $S$, then $\F_{S\cap K}(K)$ is subnormal in $\F=\F_S(G)$. So if $\F$ is realizable, then the result stated above yields that every component $\C$ of $\F$ equals $O^{p^\prime}(\E)$ for some subnormal realizable subsystem $\E$ of $\F$. The main objective of this note is to point out that the converse holds. We prove the following theorem.

\begin{theorem}\label{main}
 Let $\F$ be a saturated fusion system and write $\C_1,\dots,\C_k$ for the components of $\F$. Then $\F$ is realizable if and only if for each $i=1,\dots,r$, there exists a realizable subnormal subsystem $\E_i$ of $\F$ such that $\C_i=O^{p^\prime}(\E_i)$.
\end{theorem}

Our proof builds fundamentally on \cite[Theorem~A]{BrotoMollerOliverRuiz2023} which states that a fusion system $\F$ is realizable if it has a realizable normal subsystem $\E$ which contains all components of $\F$. It should be pointed out that both \cite[Theorem~A]{BrotoMollerOliverRuiz2023} and our Theorem~\ref{main} depend on the classification of finite simple groups.

\smallskip

Under the hypothesis of Theorem~\ref{main}, one might ask whether $\F$ is realizable if, for each $i=1,2,\dots,r$, there exists a realizable subsystem $\E_i$ of $\F$ with $\C_i=O^{p^\prime}(\E_i)$. In Section~\ref{S:Example}, we construct a family of examples to show that this is not the case in general. Thus, the hypothesis that the subsystems $\E_i$ are subnormal is actually important to conclude that $\F$ is realizable.

\smallskip

The following proposition, which is restated and proved in Proposition~\ref{P:uniqueE}, seems also interesting in the context of Theorem~\ref{main}, as it applies in particular to every component $\C$ of a saturated fusion system $\F$. The proof relies again on \cite[Theorem~A]{BrotoMollerOliverRuiz2023} and thus on the classification of finite simple groups.

\begin{proposition}\label{P:main}
Let $\C$ be a subnormal subsystem of $\F$ on $T$ with $\C=O^{p^\prime}(\C)$. Then there is a unique largest (with respect to inclusion) subnormal subsystem $\E$ of $\F$ such that $O^{p'}(\E) = \C$. If there exists a  realizable subnormal subsystem $\E'$ with $O^{p'}(\E') = \C$, then $\E$ is also realizable. 
\end{proposition}

\section{Proofs}

Throughout let $\F$ be a saturated fusion system over a finite $p$-group $S$. The reader is referred to Sections I.1-I.7 and Section~II.5 in \cite{AschbacherKessarOliver2011} for an introduction to the theory of fusion systems. We use also the following definition.

\begin{definition}
Let $\F_1,\dots,\F_k$ be saturated subsystems of $\F$ on $S_1,\dots,S_k$ respectively. Then the subsystems $\F_1,\dots,\F_k$ \emph{commute} (in $\F$) if $\F_i\subseteq C_\F(\prod_{j\neq i}S_j)$ for all $i=1,\dots,k$. 
\end{definition}

If saturated subsystems $\F_1,\dots,\F_k$  on $S_1,\dots,S_k$ commute, then it is shown in \cite[Lemma~2.21]{ChermakHenke2022} that $S_i\,\cap \,\prod_{j\neq i}S_j\leq Z(\F_i)$ for $i=1,2,\dots,k$ (a property which is not necessarily true if one drops the assumption that $\F_1,\dots,\F_k$ are saturated) and so $\F_1,\dots,\F_k$ centralize each other in the sense of \cite[Definition~2.20]{ChermakHenke2022}. The following lemma is then a consequence of Lemmas~2.16(a), 2.19(c) and 2.22 in \cite{ChermakHenke2022}.

\begin{lemma}\label{L:CentralProduct}
Let $\F_1,\dots,\F_k$ be saturated subsystems of $\F$ which commute in $\F$. Then $\F$ contains a unique saturated subsystem $\F_1*\F_2*\cdots *\F_k$ which is a central product of $\F_1,\dots,\F_k$.
\end{lemma}

We use here the definition of an (internal) central product as given in \cite[Notation~2.25, Definition~2.20]{ChermakHenke2022}), but this is by \cite[Lemma~2.19(a)]{ChermakHenke2022} equivalent to Aschbacher's notion of a  central product, which defines a central product as a quotient of the direct product of fusion systems modulo a suitable central subgroup (see Definition~2.8 and the following discussion on p.14 in \cite{AschbacherGeneralized}). 

\smallskip

The fusion system $\F$ is called \emph{simple} if it has no non-trivial proper normal subsystems. Following Aschbacher \cite[Chapter~9]{AschbacherGeneralized} we define $\F$ to be \emph{quasisimple} if $\F=O^p(\F)$ and $\F/Z(\F)$ is simple. Moreover, a \emph{component} of $\F$ is a quasisimple subnormal subsystem. We will frequently use the following lemma.

\begin{lemma}\label{L:ComponentsBasic}
Let $\E$ be a subnormal subsystem of $\F$ and let $\C_1,\dots,\C_k$ be pairwise distinct components of $\F$, which are not components of $\E$. Then $\E,\C_1,\dots,\C_k$ commute. In particular, any two distinct components of $\F$ commute.
\end{lemma}

\begin{proof}
This is stated explicitly in \cite[Lemma~7.15(b)]{ChermakHenke2022} (but similar results are implicitly contained in \cite[Chapter~9]{AschbacherGeneralized}).
\end{proof}

\begin{theorem}\label{fuscomp}
Let $G$ be a finite group, $S$ a Sylow $p$-subgroup of $G$, and $\F = \F_S(G)$. 
For each component $K$ of $G$, one of the following holds for $\K = \F_{S \cap K}(K)$. 
\begin{enumerate}
\item $\K$ is constrained,
\item $O^{p'}(\K)$ is quasisimple and so a component of $\F$. 
\end{enumerate}
\end{theorem}

\begin{proof}
This follows from \cite[Theorem~4.5]{BrotoMollerOliverRuiz2023}.
\end{proof}

The next theorem is a direct consequence of \cite[Proposition~4.9]{BrotoMollerOliverRuiz2023}, but we supply an alternative proof here which follows essentially the proof of \cite[Lemma~2.54]{HenkeLynd2022} and uses the $Z_p^*$-Theorem through \cite{HenkeSemeraro2015}. In the formulation of the proof we use that, by \cite[(6.7.1)]{AschbacherGeneralized} or \cite[Theorem~1(a)]{Henke2018}, for every normal subsystem $\E$ of $\F$, there is a subgroup $C_S(\E)$ of $S$ that is with respect to inclusion the unique maximal element of the set $\{R\leq S\colon \E\subseteq C_\F(R)\}$. Our proof depends on the classification of finite simple groups through the $Z_p^*$-Theorem for odd $p$ and through Theorem~\ref{fuscomp}

\begin{theorem}\label{T:Realizable}
Let $G$ be a finite group with $O_{p'}(G) = 1$. 
Fix a Sylow $p$-subgroup $S$ of $G$ and set $\F = \F_S(G)$. 
For each component $\C$ of $\F$, there is a unique component $K$ of $G$ such that $O^{p'}(\F_{S \cap K}(K)) = \C$. 
\label{comp fus}
\end{theorem}

\begin{proof}
Let $T$ be the Sylow group of $\C$ and set $\E = \F_{S \cap F^*(G)}(F^*(G))$. 
Since $F^*(G)$ is normal in $G$, the system $\E$ is normal in $\F$ by \cite[Proposition~I.6.2]{AschbacherKessarOliver2011}. 

\smallskip

Assume first that $\C$ is not a component of $\E$. Then $\E$ and $\C$ commute by Lemma~\ref{L:ComponentsBasic} and so $T \leq C_S(\E)$. 
But $C_S(\E) = C_S(F^*(G)) = Z(F^*(G))$ by \cite[Theorem~B]{HenkeSemeraro2015}, 
and so $T$ is abelian. By \cite[(9.1.2)]{AschbacherGeneralized} this contradicts that $\C$ is quasisimple.

\smallskip

Therefore $\C$ is a component of $\E$.
As $\E$ is the central product of $\F_{O_p(G)}(O_p(G))$ and the subsystems $\F_{S \cap K}(K)$ where $K$ runs over the components of $G$, it follows from \cite[Lemma~2.15]{HenkeLynd2022} that $\C$ is a component of $\F_{S\cap K}(K)$ for some component $K$ of $G$. The fusion system $\F_{S\cap K}(K)$ of such a component $K$ is not constrained by \cite[(9.9.1)]{AschbacherGeneralized} (or by \cite[Lemma~7.13(c)]{ChermakHenke2022}). Hence, $O^{p^\prime}(\F_{S\cap K}(K))$ is quasisimple by Theorem~\ref{fuscomp} and thus a component of $\F_{S\cap K}(K)$. As $T\leq S\cap K$ is non-abelian and any two different components of $\F_{S\cap K}(K)$ commute by Lemma~\ref{L:ComponentsBasic}, it follows that the two components $\C$ and $O^{p^\prime}(\F_{S\cap K}(K))$ must be equal.
\end{proof}

The following theorem, which is essentially a restatement of \cite[Theorem~A]{BrotoMollerOliverRuiz2023}, is needed both in the proof of Theorem~\ref{main} and in the proof of Proposition~\ref{P:main}.

\begin{theorem}\label{T:BMORSubnormal}
Suppose there exists a realizable subnormal subsystem $\E$ of $\F$ which contains every component of $\F$. Then $\F$ is realizable.
\end{theorem}

\begin{proof}
Consider a subnormal series $\E=\E_0\unlhd \E_1\unlhd\cdots\unlhd \E_n=\F$ of $\E$ in $\F$. Fix $i\in\{1,2,\dots,n\}$. Note that every component of $\E_i$ is subnormal in $\F$, therefore a component in $\F$, and thus contained in $\E=\E_0\subseteq \E_{i-1}$. Hence, $\E_{i-1}$ is a normal subsystem of $\E_i$ containing every component of $\E_i$. If $\E_{i-1}$ is realizable, it follows therefore from \cite[Theorem~A]{BrotoMollerOliverRuiz2023} that $\E_i$ is realizable. As we assume that $\E=\E_0$ is realizable, induction on $i$ yields now that $\E_i$ is realizable for $i=0,1,2,\dots,n$. In particular, $\F=\E_n$ is realizable as required.   
\end{proof}

The next lemma is a crucial ingredient in the proof of Theorem~\ref{main}. The reader might want to recall Lemma~\ref{L:CentralProduct}.

\begin{lemma}\label{L:CentralProdSubnormal}
 Let $\E_1,\dots,\E_k$ be subnormal subsystems of $\F$ and set $\C_i:=O^{p^\prime}(\E_i)$. Suppose $\C_1,\dots,\C_k$ are pairwise distinct components of $\F$. Then $\E_1,\dots,\E_k$ commute in $\F$ and the central product $\E_1*\E_2*\cdots *\E_k$ is a subnormal subsystem of $\F$.
\end{lemma}

\begin{proof}
For $i=1,2,\dots,k$ let $T_i\leq S$ such that $\E_i$ (and thus $\C_i$) is a fusion system over $T_i$. We show first for all $i,j\in\{1,2,\dots,k\}$ the following property.
 \begin{equation}\label{E:i=j}
  \mbox{If $\C_j$ is a component of $\E_i$, then $i=j$.}
 \end{equation}
For the proof assume that $\C_j$ is a component of $\E_i$. Note that $\C_i=O^{p^\prime}(\E_i)$ is a quasisimple normal subsystem of $\E_i$ and thus a component of $\E_i$. Hence, it follows from Lemma~\ref{L:ComponentsBasic} that either $\C_i$ and $\C_j$ commute or $\C_i=\C_j$. In the first case, we have in particular $[T_i,T_j]=1$ and so $T_j\leq T_i$ is abelian, which by \cite[(9.1.2)]{AschbacherGeneralized} contradicts $\C_j$ being quasisimple. So $\C_i=\C_j$ and our assumption yields $i=j$.

\smallskip

To prove the lemma we use now that, by \cite[Theorem~5.4(a),(b)]{HenkeWielandt} there exists a subnormal subsystem $\E=\la\E_1,\dots,\E_k\ra$ which is (with respect to inclusion) the smallest saturated subsystem of $\F$ in which each $\E_i$ is subnormal. Note that $\C_i=O^{p^\prime}(\E_i)$ is then also subnormal in $\E$ and thus a component of $\E$ for each $i=1,2,\dots,k$. It follows thus from Lemma~\ref{L:ComponentsBasic} and \eqref{E:i=j} that, for each $i=1,2,\dots,k$, the subsystems $\E_i,\C_1,\dots,\C_{i-1},\C_{i+1},\dots,\C_k$ commute in $\E$. This means $\E_i\subseteq C_\E(\prod_{j\neq i}T_j)$ for $i=1,2,\dots,k$ and so $\E_1,\dots,\E_k$ commute in $\E$. In particular, $\E_1,\dots,\E_k$ commute in $\F$ proving the first part of the assertion. 

\smallskip

It follows now from Lemma~\ref{L:CentralProduct} that $\E$ contains a unique saturated subsystem 
\[\D:=\E_1*\E_2*\cdots *\E_k\]
which is a central product of $\E_1,\dots,\E_k$. Note that $\D$ is (again according to Lemma~\ref{L:CentralProduct}) also the unique saturated subsystem of $\F$ which is a central product of $\E_1,\dots,\E_k$. By \cite[2.19(c), 2.21]{ChermakHenke2022}, $\E_i$ is normal in $\D$ and thus subnormal in $\D$. Since $\E$ is the smallest saturated subsystem of $\F$ in which each $\E_i$ is subnormal, it follows that $\D=\E$ is subnormal in $\F$. 
\end{proof}

\begin{proof}[Proof of Theorem~\ref{main}]
As in the theorem write $\C_1,\dots,\C_k$ for the pairwise distinct components of $\F$.

\smallskip

If $\F=\F_S(G)$ is realized by a finite group $G$ with Sylow $p$-subgroup $S$, then, identifying $S$ with its image in $G/O_{p^\prime}(G)$, we have $\F=\F_S(G)=\F_S(G/O_{p^\prime}(G))$. Hence, replacing $G$ by $G/O_{p^\prime}(G)$, we may assume $O_{p^\prime}(G)=1$. Moreover, it follows from  \cite[Proposition~I.6.2]{AschbacherKessarOliver2011} that, for every subnormal subgroup $K$ of $G$, the fusion system $\F_{S\cap K}(K)$ is subnormal in $\F$. Thus, if $\F$ is realizable, then the existence of realizable subnormal subsystems $\E_i$ with $O^{p^\prime}(\E_i)=\C_i$ for $i=1,2,\dots,k$ follows from Theorem~\ref{T:Realizable}.

\smallskip

To prove the converse direction assume now that we are given realizable subnormal subsystems $\E_1,\dots,\E_k$ of $\F$ such that $\C_i=O^{p^\prime}(\E_i)$ for $i=1,2,\dots,k$. By Lemma~\ref{L:CentralProdSubnormal}, the central product $\D:=\E_1*\E_2*\cdots *\E_k$ is subnormal in $\F$. Let $S_i\leq S$ such that $\C_i$ and $\E_i$ are fusion systems over $S_i$ for $i=1,2,\dots,k$.

\smallskip 

Recall that the components $\C_1,\dots,\C_k$ of $\F$ are all contained in $\D$. Thus, by Theorem~\ref{T:BMORSubnormal}, it is sufficient to argue that $\D$ is realizable. For $i=1,2,\dots,k$ choose a finite group $G_i$ with Sylow $p$-subgroup $S_i$ such that $\E_i=\F_{S_i}(G_i)$. It is an easy exercise to check that the direct product $G:=G_1\times G_2\times \cdots \times G_k$ of groups realizes the direct product $\E:=\E_1\times \E_2\times \cdots \times \E_k$ of fusion systems. Moreover, setting $Z_i:=Z(\E_i)$, it follows from \cite[2.19(a), 2.21]{ChermakHenke2022} that $\alpha\colon S_1\times S_2\times\cdots \times S_k\rightarrow S$ is a group homomorphism which induces an epimorphism $\E\rightarrow \D$ with $\ker(\alpha)\leq Z_1\times Z_2\times \cdots \times Z_k$. 
By \cite[Proposition~I.5.4]{AschbacherKessarOliver2011}, we have $\E_i=\F_{S_i}(C_{G_i}(Z_i))$. Thus, replacing $G_i$ by $C_{G_i}(Z_i)$ we may assume $Z_i\leq Z(G_i)$ and thus 
\[Z:=\ker(\alpha)\leq Z_1\times Z_2\times \cdots \times Z_k\leq Z(G).\]
Now \cite[Example~II.5.6]{AschbacherKessarOliver2011} (or an easy direct argument) yields that $\E/Z$ is realizable by $G/Z$ and so $\D\cong \E/Z$ is also realizable by a finite group as required. 
\end{proof}

Note that Proposition~\ref{P:main} is implied by the following proposition. The proof of the last sentence in the proposition depends on the classification of finite simple groups through \cite[Theorem~A]{BrotoMollerOliverRuiz2023}.

\begin{proposition}\label{P:uniqueE}
Let $\C$ be a subnormal subsystem of $\F$ on $T$ with $\C=O^{p^\prime}(\C)$ and set 
\[\mathfrak{S}:=\{\D\subn\F\colon O^{p^\prime}(\D)=\C\}.\] 
Then there exists a unique largest (with respect to inclusion) member $\E$ of $\mathfrak{S}$, and every subsystem in $\mathfrak{S}$ is subnormal in $\E$. Moreover, if some member of $\mathfrak{S}$ is realizable, then $\E$ is also realizable. 
\end{proposition}

\begin{proof}
Note that $\mathfrak{S}$ is non-empty as $\C\in\mathfrak{S}$. Observe also that each member of $\mathfrak{S}$ is a subsystem over $T$. Hence, by \cite[Theorem~5.4(a),(b)]{HenkeWielandt}, there is a saturated subsystem $\E$ of $\F$ over $T$ that is minimal under inclusion subject to containing each member of $\mathfrak{S}$ as a subnormal subsystem, and $\E$ is subnormal in $\F$. 

\smallskip

Since $\C\in\mathfrak{S}$ is subnormal in $\E$, we may consider a subnormal series $\C = \C_0 \norm \C_1 \norm \cdots \norm \C_n = \E$. As $\C_0,\C_1,\dots,\C_n$ are all subsystems on $T$, it follows from \cite[Theorem~7.56]{CravenTheory} that $O^{p'}(\C_{i}) \subseteq \C_{i-1}$ for each $i = 1,\dots,n$, and so $O^{p'}(\C_{i-1}) = O^{p'}(\C_{i})$. This implies inductively that $O^{p'}(\E) = \C$. Hence, $\E\in\mathfrak{S}$ and so $\E$ is a maximal member of $\mathfrak{S}$.

\smallskip

Suppose now that some member $\E'$ of $\mathfrak{S}$ is realizable. As $\C=O^{p^\prime}(\E)\subseteq \E'$, it follows from \cite[Theorem~7.10(d)]{ChermakHenke2022} that $\E'$ contains every component of $\E$. Hence, $\E$ is realizable by Theorem~\ref{T:BMORSubnormal}.
\end{proof}

\section{Examples}\label{S:Example}

One direction of Theorem~\ref{main} states that a saturated fusion system $\F$ with components $\C_1,\dots,\C_r$ is realizable if, for each $i=1,2,\dots,r$, there exists a realizable subnormal subsystem $\E_i$ of $\F$ with $O^{p^\prime}(\E_i)=\C_i$. It is natural to ask whether one could drop the assumption here that $\E_i$ is subnormal and just assume that $\E_i$ is a realizable subsystem of $\F$ with $O^{p^\prime}(\E_i)=\C_i$. We give now some examples to show that this is not the case. Indeed, for each integer $r\geq 1$, we construct a counterexample with precisely $r$ components. 

\smallskip

If $C$ is a group or a fusion system and $r\geq 1$ an integer, then we write $C^r=C\times C\times \cdots \times C$ for the $r$-fold direct product of $C$ with itself. If $S$ is a group, $A\leq \Aut(S)$ and $U\leq S$, then we set $A|_U:=\{\beta|_U\colon \beta\in A\}$. 

\begin{example}
 Let $p\geq 5$ be a prime and $H$ a finite simple group such that, for a Sylow $p$-subgroup $T$ of $H$, the fusion system $\C:=O^{p^\prime}(\F_T(H))$ is an exotic simple fusion system.  Thus, $H$ is one of the groups listed in \cite[Theorem~A(c)]{OliverRuiz2021}. Fix furthermore an integer $r\geq 2$. Let $S:=T^r$ so that $\E:=\C^r$ is a fusion system over $S$. Define
 \[A:=\{(\alpha,\alpha,\dots,\alpha)\colon \alpha\in \Aut_H(T)\}\leq \Aut(S).\]
 Let $\iota_i\colon T\rightarrow S$ be the $i$th inclusion map. Set 
 \[T_i:=T\iota_i,\;\C_i:=\C^{\iota_i}\mbox{ and }\E_i:=\<\C_i,A|_{T_i}\>\mbox{ for }i=1,2,\dots,r.\]
 Then the following hold.
 \begin{itemize}
 \item[(a)] The fusion system $\F:=\<\E,A\>$ is saturated and exotic. It has components $\C_1,\C_2,\dots,\C_r$, and each component is isomorphic to $\C$. For each $i=1,2,\dots,r$, the fusion system $\E_i$ is a saturated subsystem of $\F$ with $O^{p^\prime}(\E_i)=\C_i$. Moreover, $\E_i$ is isomorphic to $\F_T(H)$ and thus realizable.
 \item[(b)] Set $R:=T_2T_3\cdots T_r$ and $\G=N_\F(R)$. Then $\G$ is saturated and exotic. Its only component is $\C_1$. Moreover, $\E_1$ is a realizable subsystem of $\G$ with  $O^{p^\prime}(\E_1)=\C_1$.  
 \end{itemize}
\end{example}

\begin{proof}
\textbf{(a)} It is shown in \cite[Theorem~I.6.6]{AschbacherKessarOliver2011} that the direct product of two saturated fusion systems is saturated. As this result can easily be generalized to direct products with arbitrarily many factors, it follows that the fusion system $\E$ is saturated. Observe that every automorphism in $A$ is a fusion-preserving automorphism of $\E$.
Moreover, $A\cong \Aut_H(T)$ and 
\[\pi:=A\Aut_\E(S)/\Aut_\E(S)\cong A/(A\cap \Aut_\E(S))\cong \Aut_H(T)/\Aut_\C(T)\]
is a $p^\prime$-group. It follows thus from \cite[Theorem~5.7(a)]{BCGLO2007} applied with $\E$ in place of $\F$ that $\F$ is saturated. Since the condition (b) in \cite[Proposition~I.6.4]{AschbacherKessarOliver2011} is satisfied, it follows from that proposition that $\E$ is $\F$-invariant (in the sense defined \cite[Definition~I.6.1]{AschbacherKessarOliver2011}). As $\E$ is a saturated subsystem of $\F$ over $S$, this implies that $\E$ is normal in $\F$.

\smallskip

By assumption and construction, $\C_i\cong \C$ is a simple fusion system for $i=1,2,\dots,r$. Moreover, $\C_i\unlhd \E\unlhd \F$ is subnormal in $\F$. Hence, $\C_1,\C_2,\dots,\C_r$ are components of $\F$. By \cite[(9.1.2)]{AschbacherGeneralized}, every component of $\F$ is a fusion system over a non-abelian subgroup of $S$. As $S=T_1T_2\cdots T_r$, it follows thus from Lemma~\ref{L:ComponentsBasic} that $\C_1,\C_2,\cdots,\C_r$ are the only components of $\F$.

\smallskip

For $\alpha\in \Aut_H(T)$ and $i=1,2,\dots,r$, we have $(\alpha,\alpha,\dots,\alpha)|_{T_i}=\alpha^{\iota_i}$. Thus, $A|_{T_i}=\Aut_H(T)^{\iota_i}$. As $\F_S(H)=\<\C,\Aut_H(T)\>$, it follows that $\E_i=\F_S(H)^{\iota_i}\cong \F_S(H)$. In particular, $\E_i$ is realizable and $O^{p^\prime}(\E_i)=\C^{\iota_i}=\C_i$. 

\smallskip

To prove (a), it remains to show that $\F$ is exotic. Note that $R:=T_2T_3\cdots T_r$ is normal in $S$ and thus fully $\F$-normalized. Therefore, it is a consequence of \cite[Proposition~I.5.4]{AschbacherKessarOliver2011} that $\G:=N_\F(R)$ is realizable if $\F$ is realizable. Hence, it is sufficient to prove that $\G$ is exotic, which we will do below in the proof of (b).

\smallskip

\textbf{(b)} As observed before, $R\unlhd S$ is fully $\F$-normalized. Hence, $\G=N_\F(R)$ is saturated by \cite[Theorem~I.5.5]{AschbacherKessarOliver2011}. As $\E\unlhd \F$, it follows from \cite[Lemma~2.13]{ChermakHenke2022} that  $N_\E(R)\unlhd N_\F(R)=\G$. Using the construction of $\E$, one observes furthermore that $\C_1\unlhd N_\E(R)$. Hence, $\C_1$ is subnormal in $\G$. As $\C_1\cong \C$ is simple, it follows that $\C_1$ is a component of $\G$. 

\smallskip

Notice now that $\F_R(R)\unlhd \G$. If there were a component $\K$ of $\G$ with $\K\neq \C_1$, then $\K$ would thus commute with $\C_1$ and $\F_R(R)$ by Lemma~\ref{L:ComponentsBasic}. As $S=T_1R$, this would however mean that $\K$ is a subsystem over an abelian subgroup of $S$, contradicting \cite[(9.1.2)]{AschbacherGeneralized}.  Hence, $\C_1$ is the only component of $\G$. Observe that $\E_1$ is a subsystem of $\G$. As seen in (a), $\E_1$ is realizable and $O^{p^\prime}(\E_1)=\C_1$. 

\smallskip

It remains now to show that $\G$ is exotic. Assume that this is false, i.e. $\G=\F_S(G)$ for some finite group $G$ containing $S$ as a Sylow $p$-subgroup. As $\G=N_\G(R)$ and $R\unlhd S$, it follows from \cite[Proposition~I.5.4]{AschbacherKessarOliver2011} that $\G=\F_S(N_G(R))$. Thus, replacing  $G$ by $N_G(R)$, we may assume that $R\unlhd G$. Moreover, replacing $G$ by $G/O_{p^\prime}(G)$, we may assume $O_{p^\prime}(G)=1$. Then by Theorem~\ref{T:Realizable}, there exists a component $K$ of $G$ with $K\cap S=T_1$ and $O^{p^\prime}(\F_{T_1}(K))=\C_1$. 

\smallskip

As $\C_1\cong \C$ is exotic, we have $\C_1\neq \F_{T_1}(K)$. Hence, there exists $\beta\in\Aut_K(T_1)\backslash\Aut_{\C_1}(T_1)$.  As $R\unlhd G$, we have $[K,R]=1$. Thus, $\beta$ extends to a $\hat{\beta}\in\Aut_\G(S)$ such that $\hat{\beta}$ restricts to the identity on $R$. Observe that $\Aut_\F(S)=\Aut_\E(S)A$. Hence, it follows from the Frattini condition for $\F$-invariant subsystems (cf. \cite[Definition~I.6.1]{AschbacherKessarOliver2011}) that $\hat{\beta}=\phi\tilde{\alpha}$ for some $\phi\in\Aut_\E(S)$ and some $\tilde{\alpha}\in A$. Write $\phi=(\phi_1,\dots,\phi_r)$ and $\tilde{\alpha}=(\alpha,\dots,\alpha)$ where $\phi_i\in\Aut_{\C}(T)$ for $i=1,2,\dots,r$ and $\alpha\in \Aut_H(T)$. Then 
\[\hat{\beta}=(\phi_1\alpha,\phi_2\alpha,\dots,\phi_r\alpha).\]
As $\hat{\beta}$ restricts to the identity on $T_2$, we have $\phi_2\alpha=\id_T$ and hence $\alpha=\phi_2^{-1}\in\Aut_\C(T)$. Thus, $\phi_1\alpha\in\Aut_\C(T)$ and $\beta=\hat{\beta}|_{T_1}=(\phi_1\alpha)^{\iota_1}\in\Aut_\C(T)^{\iota_1}=\Aut_{\C_1}(T_1)$, contradicting the choice of $\beta$. This shows that $\G$ is exotic and completes the proof.  
\end{proof}

\bibliographystyle{amsalpha}
\bibliography{mybib.bib}
\end{document}